\documentclass[12pt]{arxivclassfile}

\usepackage{mathrsfs}

\def\R{{\mathbb R}}

\def\d{{\rm d}}
\def\e{{\rm e}}
\def\:{\colon}

\def\email#1{\ead{#1}}

\usepackage{graphicx,epsfig,psfrag}

\usepackage{mathscinet}
\usepackage{url}
\usepackage{amssymb}
\usepackage{amsmath}
\usepackage{amsthm}
\usepackage{enumerate}
\newtheorem{theorem}{Theorem}[section]

\newtheorem{corollary}[theorem]{Corollary}
\newtheorem{lemma}[theorem]{Lemma}

\newtheorem{definition}[theorem]{Definition}

\def\dist{{\rm dist}}

\begin{document}

\begin{frontmatter}

\author[RL]{R.\ Laister\corref{thing}}
\ead{Robert.Laister@uwe.ac.uk}
\address[RL]{Department of Engineering Design and Mathematics, \\ University of the West of England, Bristol BS16 1QY, UK.}
\author[warwick]{J.C.\ Robinson}
\ead{J.C.Robinson@warwick.ac.uk}
\author[warwick]{M.\ Sier{\.z}\polhk{e}ga}
\ead{M.L.Sierzega@warwick.ac.uk}
\address[warwick]{Mathematics Institute, Zeeman Building,\\ University of Warwick, Coventry CV4 7AL, UK.}
%
\author[AVL]{A.\ Vidal-L\'opez}
\email{Alejandro.Vidal@xjtlu.edu.cn}
\address[AVL]{Department of Mathematical Sciences, Xi'an Jiaotong-Liverpool University, Suzhou 215123, China P.~R.}

\title{A complete characterisation of local existence for semilinear heat equations in Lebesgue spaces}

\begin{abstract}
We consider the scalar semilinear heat equation $u_t-\Delta u=f(u)$, where $f\colon[0,\infty)\to[0,\infty)$ is continuous and non-decreasing but need not be convex. We completely characterise those functions $f$ for which the equation has a local solution bounded in $L^q(\Omega)$ for all non-negative initial data $u_0\in L^q(\Omega)$, when $\Omega\subset\R^d$ is a bounded domain with Dirichlet boundary conditions. For $q\in(1,\infty)$ this holds if and only if $\limsup_{s\to\infty}s^{-(1+2q/d)}f(s)<\infty$; and for $q=1$ if and only if $\int_1^\infty s^{-(1+2/d)}F(s)\,\d s<\infty$, where $F(s)=\sup_{1\le t\le s}f(t)/t$. This shows for the first time that the model nonlinearity $f(u)=u^{1+2q/d}$ is truly the `boundary case' when $q\in(1,\infty)$, but that this is not true for $q=1$.

The same characterisation results hold for the equation posed on the whole space $\R^d$ provided that in addition $\limsup_{s\to0}f(s)/s<\infty$.
\end{abstract}


\end{frontmatter}

\section{Introduction}

This paper concerns local existence of solutions of the scalar semilinear heat equation
\begin{equation}\label{nhe}
u_t-\Delta u=f(u),\qquad u(0)=u_0\ge0,
\end{equation}
on the whole space $\R^d$ and on
smooth bounded domains $\Omega\subset\R^d$ with Dirichlet boundary conditions, when $u_0\in L^q(\Omega)$, $1\le q<\infty$.
Throughout and without loss of generality we assume that $\Omega$ contains the origin.

 We give a complete solution to the classical problem of characterising those functions $f$ for which (\ref{nhe}) has a local solution that is bounded in $L^q(\Omega)$ for all non-negative initial data in $L^q(\Omega)$. It is perhaps surprising that such results are not already available in the literature, but they are not; nor do our characterisations follow from what has previously been proved about (\ref{nhe}). Indeed, most previous results focus on the particular nonlinearity $f(u)=u^p$, with more general treatments assuming that $f$ is convex. We impose no such restrictions in this paper, requiring only that $f:[0,\infty)\to[0,\infty)$ is continuous and non-decreasing.

 The main contribution this paper makes is in identifying the correct characterisation for both the case $q>1$ and $q=1$. Given the `correct' assumptions on $f$, the methods of proof for existence/non-existence are not difficult, but still require some care. The non-existence results rely on lower bounds on the heat kernel, and in particular on lower bounds for the action of the heat semigroup on initial conditions equal to the characteristic function of a ball. In a very imprecise way, they show that for $q>1$ the model equation with $f(u)=u^p$ `tells the whole story', but that this is decidedly not the case when $q=1$.

Local well-posedness of (\ref{nhe}) for smooth data falls within the scope of the standard theory of parabolic equations that goes back half a century \cite{LSU}. In the early 1980s the well-posedness theory was extended by Weissler \cite{W79,W80,W81} to include initial data in Lebesgue spaces, with a locally Lipschitz source term $f$ satisfying a Lipschitz bound of the form
\begin{equation}\label{Lip}
|f(u)-f(v)|\leq C |u-v|(1+|u|^{p-1}+|v|^{p-1})
\end{equation}
providing sufficient conditions for local existence (and uniqueness).

In particular, in these papers and in much subsequent work, attention was almost exclusively focused on the canonical model with $f(u)=|u|^{p-1}u$ introduced by Fujita \cite{Fuj1}. For this particular nonlinearity, given $q\in(1,\infty)$, the pioneering results of \cite{W79,W80,W81} along with those of Giga \cite{G} and Brezis \& Cazenave \cite{BC} identify a critical exponent $p^\star=1+2q/d$ such that (\ref{nhe}) with $f$ satisfying (\ref{Lip}) is locally well posed in $L^q$ if and only if $p\le p^\star$; for $p>p^\star$ one can find initial data in $L^q$ for which there is no local solution. While for $q>1$ the equation is well behaved when $p=p^\star$, for the case $q=1$ Celik \& Zhou \cite{CZ} showed that for the critical exponent $p^\star=1+2/d$ there are $L^1$ initial data for which there is no solution (resolving a problem posed in \cite{BC}).

This theory has been extended in a number of ways. One natural direction was to extend the theory towards weaker classes of data (e.g.\ measure-valued initial conditions), see Brezis \& Friedman \cite{BF}, for example. Along these lines, Baras \& Pierre \cite{BP} obtained a necessary and sufficient condition on the initial condition for local existence of solutions when $f$ is convex.

A second direction focuses on finite-time blowup versus global existence. In most of these analyses, the particular form of the Fujita nonlinearity $f(u)=|u|^{p-1}u$ or a related convexity assumption plays a crucial role, see for example \cite{BaC, BCMR, FM, Fuj1, GV, Kaplan, NS}. For example, the homogeneity of $u^p$ facilitates the use of similarity solutions - such scale invariance also makes transparent the role of the critical exponent, while for a general convex $f$ one can use Jensen's inequality. Note that we do not consider finite-time blowup here, but rather local non-existence, i.e.\ `immediate blowup' in some sense.

However, most of these results break down if we only make the assumption that $f$ is monotonic.  In this case, in order to describe fully the conditions on $f$ ensuring that an initial condition in $L^{q}$ gives rise to a local solution we need a better understanding of  the delicate balance between the smoothing action of the heat flow and the converse effect of the growing source. In this paper we provide, for every $q\in[1,\infty)$, a precise characterisation of those $f$ for which the equation (\ref{nhe}) has local solutions bounded in $L^q(\Omega)$ for all non-negative initial data $u_0\in L^q(\Omega)$. Note that this includes the delicate case $q=1$.

First we show that for $q\in[1,\infty)$, if
\begin{equation}\label{introLq}
\limsup_{s\to\infty}s^{-(1+2q/d)}f(s)=\infty
\end{equation}
then there exists a non-negative $u_0\in L^q(\Omega)$ for which equation (\ref{nhe}) has no local solution that is bounded in $L^q(\Omega)$. Since the existence of a finite limit in (\ref{introLq}) implies that $f(s)\le C(1+s^{1+2q/d})$ for some constant $C$, monotonicity of solutions along with classical results for (\ref{u2p}) yields local existence in this case for $q\in(1,\infty)$. It follows (see Theorem \ref{thm:char1}) that equation (\ref{nhe}) has at least one local $L^q$-bounded solution for every non-negative $u_0\in L^q(\Omega)$ if and only if
$$
\limsup_{s\to\infty}s^{-(1+2q/d)}<\infty.
$$
The `moral' of this is that for $q\in(1,\infty)$, the model problem with $f(s)=s^p$ in some sense tells the whole story, since the critical case lies precisely on the boundary between local existence/non-existence. (This idea has perhaps always been implicit in the discussions in the literature, but has not had a rigorous proof until now.)

The case $q=1$ is more delicate, and is well known to be significantly more challenging. As remarked above, Celik \& Zhou \cite{CZ} showed that for the canonical equation
 \begin{equation}\label{u2p}
 u_t-\Delta u=u^p
 \end{equation}
  with $p=p^\star=1+2/d$ there is non-negative initial data in $L^1(\R^d)$ and $L^1(\Omega)$ for which there is no local solution. One might therefore conjecture that for $q=1$ the condition in (\ref{introLq}) can be weakened to
$$
\limsup_{s\to\infty}s^{-(1+2/d)}f(s)>0
$$
(i.e.\ the limit is finite but strictly positive) and still ensure non-existence for some non-negative $u_0\in L^1(\Omega)$. In fact significantly more is true: we show that the condition
$$
\sum_{k=1}^\infty s_k^{-(1+2/d)}f(s_k)=\infty
$$
for some sequence such that $s_{k+1}\ge\theta s_k$ ($\theta>1$) is sufficient for such a non-existence result. In particular, if $f$ satisfies this condition there are non-negative data in $L^1(\Omega)$ for which there is no solution with $u(t)\in L^1(\Omega)$ for any $t>0$.

For any particular $f$ this divergent series condition seems awkward to check in practice, so we show that it is equivalent to the integral condition
\begin{equation}\label{withFs}
\int_1^\infty s^{-(1+2/d)}F(s)\,\d s=\infty,\qquad\mbox{where}\quad F(s)=\sup_{1\le t\le s}\frac{f(t)}{t}.
\end{equation}
Remarkably, if the integral in (\ref{withFs}) is finite, then a version of an argument due to Sier{\.z}\polhk{e}ga \cite{MT} guarantees local existence of an $L^1$-bounded  solution for every non-negative $u_0\in L^1(\Omega)$ (in fact the solution is in $L^\infty(\Omega)$ for every $t>0$). As a consequence we obtain our second main result (Corollary \ref{corL1}), namely that equation (\ref{nhe}) has at least one local $L^1$-bounded solution for every non-negative $u_0\in L^1(\Omega)$ if and only if
$$
\int_1^\infty s^{-(1+2/d)}F(s)\,\d s<\infty,\qquad\mbox{where}\quad F(s)=\sup_{1\le t\le s}\frac{f(t)}{t}.
$$
Here the `moral' is that the model problem {\it does not} tell the whole story.

We note here that we do not treat the question of uniqueness in this paper, but concentrate solely on local existence. For this reason we do not require any Lipschitz-type assumptions on $f$ (such as (\ref{Lip})).

The paper is organised as follows. In Section \ref{sec:heat} we prove some preliminary lower bounds on solutions of the heat equation for an initial condition that is the characteristic function of a ball; these are the key estimates that we use in our proofs. Section 3 contains the results for $q>1$, with Section 4 treating $q=1$. In Section 5 we discuss the problem posed on the whole space and on a bounded domain with Neumann boundary conditions, and we end with a brief recapitulation and discussion of open problems.

\section{Lower bounds on solutions of the Dirichlet heat equation}\label{sec:heat}

An important ingredient of our arguments is the following simple lemma, which gives a lower bound on the action of the heat equation on the characteristic function of a Euclidean ball. We write $B_r(x)$ for the ball in $\R^d$ of radius $r$ centred at $x$, denote by $\chi_{r}$ the characteristic function of $B_r:=B_r(0)$, and use $\omega_d$ for the volume of the unit ball in $\R^d$.

The solution of the heat equation on $\Omega$ with Dirichlet boundary conditions,
$$
u_t-\Delta u=0,\qquad u|_{\partial\Omega}=0,\qquad u(x,0)=u_0(x)\in L^1(\Omega)
$$
can be given in terms of the Dirichlet heat kernel $K_\Omega$ by the expression
$$
u(x,t)=[S(t)u_0](x):=\int_\Omega K_\Omega(x,y;t)u_0(y)\,\d y.
$$
The proofs of the results in this section use
 the following lower bound on $K_\Omega$: if the line segment joining $x$ and $y$ is a distance at least $\delta$ from $\partial\Omega$, then the Dirichlet heat kernel $K_\Omega(x,y;t)$ is bounded below by the Gaussian heat kernel on $\R^d$,
\begin{equation}\label{KOG}
K_\Omega(x,y;t)\ge\e^{-d^2\pi^2t/4\delta^2}(4\pi t)^{-d/2}\e^{-|x-y|^2/4t}\qquad\mbox{for all}\quad t>0.
\end{equation}
(See van den Berg \cite{vdB1}, Theorem 2 and Lemmas 8 and 9.)

\begin{lemma}\label{cflb}
  There exists an absolute constant $c_d\in(0,1)$, which depends only on $d$,  such that for any $r>0$ for which $B_{r+\delta}\subset\Omega$,
  \begin{equation}\label{grows}
  S(t)\chi_{r}\ge c_d\,\left(\frac{r}{r+\sqrt t}\right)^d\,\chi_{r+\sqrt t},
  \end{equation}
  for all $0<t\le\delta^2$.
\end{lemma}

\begin{proof}
For $x$ such that $\dist(x,\partial\Omega)\ge\delta$, the lower bound in (\ref{KOG}) implies that for $0<t\le\delta^2$
\begin{align*}
[S(t)\chi_r](x)&=\int_{B_r(0)}K_\Omega(x,y;t)\,\d y\\
&\ge \e^{-d^2\pi^2/4}\,(4\pi t)^{-d/2}\int_{B_r}\e^{-|x-y|^2/4t}\,\d y.
\end{align*}
The latter integral is radially symmetric and decreasing with $|x|$ and so for $|x|\le r+\sqrt{t}$, choosing any unit vector $\bf u$ we can write
\begin{align*}
[S(t)\chi_r](x)&\ge \e^{-d^2\pi^2/4}\,(4\pi t)^{-d/2}\int_{B_r((r+\sqrt t){\bf u})}\e^{-|z|^2/4t}\,\d z\\
&=\e^{-d^2\pi^2/4}\,\pi^{-d/2}\int_{B_{r/2\sqrt t}((\frac{1}{2}+\frac{r}{2\sqrt{t}}){\bf u})}\e^{-|w|^2}\,\d w.
\end{align*}
Observing that
$$
B_{r/2\sqrt t}(({\textstyle\frac{1}{2}}+{\textstyle\frac{r}{2\sqrt{t}}}){\bf u})\subseteq B_{\rho/2\sqrt t}(({\textstyle\frac{1}{2}}+{\textstyle\frac{\rho}{2\sqrt{t}}}){\bf u})
$$
if $\rho\ge r$, it follows that for $r/\sqrt t\ge 1$ we have
$$
[S(t)\chi_r](x)\ge\e^{-d^2\pi^2/4}\,\pi^{-d/2}\int_{B_{1/2}({\bf u})}\e^{-|w|^2}\,\d w=:c'_d.
$$
On the other hand, if $r/\sqrt t\le 1$ then
$$
[S(t)\chi_r](x)\ge\e^{-d^2\pi^2/4}\,\pi^{-d/2}\left(\frac{r}{2\sqrt t}\right)^d\e^{-9/4}=:c''_d(r/\sqrt t)^d.
$$
So with $c_d=\min(c_d',c_d'')$
$$
[S(t)\chi_r](x)\ge c_d\left(\frac{r}{\max(r,\sqrt t)}\right)^d\ge c_d\left(\frac{r}{r+\sqrt t}\right)^d.\eqno\qedhere
$$
\end{proof}

We will use this result in the form of one of the following two simple corollaries.

\begin{corollary}\label{norm1}
  There exists an absolute constant $\alpha_d>0$, depending only on $d$,  such that for any $r,\delta>0$ for which $B_{r+\delta}\subset\Omega$,
$$
  \int_\Omega S(t)\chi_{r}\,\d x\ge \alpha_dr^d,
$$
  for all $0<t\le\delta^2$.
\end{corollary}

\begin{proof}
  Integrating the inequality in (\ref{grows}) over $\Omega$ yields
$$
  \int_\Omega S(t)\chi_r\,\d x
  \ge c_d\,\left(\frac{r}{r+\sqrt t}\right)^d\,\int_\Omega\chi_{r+\sqrt t}\,\d x\\
  =c_d\omega_dr^d.\eqno\qedhere
$$
\end{proof}

\begin{corollary}\label{notail}
  There exists an absolute constant $\beta_d>0$, depending only on $d$,  such that for any $r,\delta>0$ for which $B_{r+\delta}\subset\Omega$,
$$
  S(t)\chi_{r}\ge \beta_d\,\chi_{r+\sqrt t},
$$
  for all $0<t\le\min(\delta^2,r^2)$.
\end{corollary}

\section{Initial data in $L^q(\Omega)$, $q\in(1,\infty)$}

Given these preliminaries we can prove our first non-existence result. We take the following definition from \cite{SQ} as our (essentially minimal) definition of a solution of (\ref{nhe}). Note that any classical or mild solution is a local integral solution in the sense of this definition \cite[p.\ 77--78]{SQ}.

\begin{definition}\label{lis}
Given $f:[0,\infty)\to[0,\infty)$  and $u_0\ge0$ we say that $u$ is a \emph{local integral solution} of (\ref{nhe}) on $[0,T)$ if $u:\Omega \times[0,T)\to[0,\infty]$ is measurable, finite almost everywhere, and satisfies
\begin{equation}
u(t)=S(t)u_0+\int_0^t S(t-s)f(u(s))\, \d s\label{eq:mild}
\end{equation}
 almost everywhere in $\Omega \times[0,T)$.\label{def:soln}
\end{definition}

We will be interested in solutions with non-negative initial data $u_0\in L^q(\Omega)$ that remain bounded in $L^q(\Omega)$. To this end we make the following definition.

\begin{definition}
  We say that $u$ is a \emph{local $L^q$ solution} of (\ref{nhe}) if $u$ is a local integral solution on $[0,T)$ for some $T>0$ and $u\in L^\infty((0,T);L^q(\Omega))$. If every non-negative $u_0\in L^q(\Omega)$ gives rise to a local $L^q$ solution then we say that (\ref{nhe}) has the \emph{local existence property in $L^q(\Omega)$}.
\end{definition}

We now show that there are non-negative initial conditions in $L^q(\Omega)$ for which there is no local $L^q$ solution if $f$ satisfies the asymptotic growth condition in (\ref{infinite1}). This condition is modelled on the stronger condition
$$
\limsup_{s\to\infty}s^{-\gamma}f(s)=\infty
$$
for some $\gamma>q(1+2/d)$, which was used by Laister et al.\ in \cite{LRS1} to construct a non-negative initial condition in $L^q(\Omega)$ for which any local integral solution is not in $L^1_{\rm loc}(\Omega)$ for any $t>0$ small (a stronger form of non-existence than we obtain in Theorem \ref{T1}). (A similar condition was used to analyse the problem on the whole space in \cite{LRS0}.)

Note that our assumption in Theorem \ref{T1} does not imply any lower bounds on the function $f$ itself, nor does our result require them; e.g.\ we impose no condition on the behaviour of
$$
\liminf_{s\to\infty}s^{-(1+2q/d)}f(s),
$$
as in Weissler \cite{W80} (Theorem 5, Corollaries 5.1 and 5.2), nor do we require $f$ to be continuous.

\begin{theorem}\label{T1}
Let $f:[0,\infty)\to[0,\infty)$ be non-decreasing. If $q\in[1,\infty)$ and
\begin{equation}\label{infinite1}
\limsup_{s\to\infty}s^{-(1+2q/d)}f(s)=\infty
\end{equation}
then there exists a non-negative $u_0\in L^q(\Omega)$ such that
\begin{equation}\label{eqn}
u_t-\Delta u=f(u),\qquad u|_{\partial\Omega}=0,\qquad u(0)=u_0
\end{equation}
has no local $L^q$ solution.
\end{theorem}

\begin{proof} Set $p=1+(2q/d)$. It follows from (\ref{infinite1}) that
we can choose a sequence $\phi_k$ such that
$$
 \phi_k\ge k\qquad\mbox{and}\qquad f(\phi_k)\ge \phi_k^p\e^{k/q}.
$$

We now construct an initial condition in $L^q(\Omega)$ that is the sum of characteristic functions on a sequence of balls of decreasing radius. More precisely, set
$$
 r_k=\varepsilon\phi_k^{-q/d}k^{-2q/d},
$$
 and choose the initial data
$$
 u_0(x)=\sum_{k=1}^\infty u_k,\qquad u_k=\beta_d^{-1}\phi_k\chi_{r_k},
$$
 where $\beta_d$ is the constant from Corollary \ref{notail} and $\varepsilon$ is chosen sufficiently small that $B_{2r_k}\subset\Omega$ for every $k$.   Noting that $\|u_k\|_{L^q}=\beta_d^{-1}\varepsilon^{d/q}k^{-2}$ it follows that
$$
 \|u_0\|_{L^q}\le\sum_k\|u_k\|_{L^q}= \beta_d^{-1}\varepsilon^{d/q}\sum_k k^{-2}<\infty.
$$

Now, if a solution $u(t)$ of (\ref{eqn}) exists, then it can be written using the variation of constants formula,
\begin{equation}\label{voc}
u(t)=S(t)u_0+\int_0^t S(t-s)f(u(s))\,\d s.
\end{equation}
Since $u\ge0$ and $f\ge0$, it is immediate that
\begin{equation}\label{oneball}
u(t)\ge S(t)u_0\ge S(t)u_k,
\end{equation}
for any choice of $k$. Choosing and fixing one $k$ for now, we can neglect the first term in (\ref{voc}) and use the lower bound in (\ref{oneball}) to obtain
\begin{equation}\label{easylb}
u(t)\ge \int_0^t S(t-s)f(S(s)u_k)\,\d s,
\end{equation}
since $f$ is non-decreasing. To aid readability, and in a slight abuse of notation, we now write $\chi(r)$ for $\chi_r$.

Corollary \ref{notail} with $\delta=r_k$ implies that
$$
S(s)u_k=S(s)[\beta_d^{-1}\phi_k\chi(r_k)]\ge\phi_k\,\chi(r_k+\sqrt t)\ge \phi_k\,\chi(r_k),\qquad 0\le s\le r_k^{2},
$$
and so
$$
f(S(s)u_k)\ge f(\phi_k)\,\chi(r_k),\qquad 0\le s\le r_k^{2},
$$
since $f$ is non-decreasing. Using Corollary \ref{notail} again
$$
S(t-s)f(S(s)u_k)\ge \beta_d f(\phi_k)\,\chi(r_k),\qquad 0\le s\le t\le r_k^2.
$$

Now, using the lower bound in (\ref{easylb}), it follows that for any\footnote{We prove a lower bound valid for $t$ in an interval since a priori our definition of a local $L^q$ solution requires only that $u(t)\in L^q$ for almost every $t$.}  $t\in[t_k/2,t_k]$, where $t_k=r_k^2$,
$$
[u(t)](x)\ge\int_0^{t}S(t-s)f(S(s)u_0)\,\d s\ge \beta_d\,r_k^2\,f(\phi_k)\,\chi(r_k).
$$
Thus
\begin{align*}
\|u(t)\|_{L^q}^q\ge\int_{B(r_k)}|u(t)|^q\,\d x&\ge c\,r_k^{2q}f(\phi_k)^q r_k^d \\
&=c r_k^{d+2q}f(\phi_k)^q\\
&\ge c [\varepsilon\phi_k^{-q/d}k^{-2q/d}]^{d+2q}\phi_k^{[1+(2q/d)]q}\e^k\\
&= ck^{-2q(d+2q/d)}\e^k.
\end{align*}
Since the right-hand side tends to infinity as $k\to\infty$, it follows that $u$ is not an element of $L^\infty((0,T);L^q(\Omega))$ for any $T>0$.
\end{proof}

We remarked above that Laister et al. \cite{LRS1} showed that under the stronger condition
$$
\limsup_{s\to\infty}s^{-\gamma}f(s)=\infty,\qquad\gamma>q(1+2/d)
$$
there is non-negative initial data in $L^q(\Omega)$ for which any local solution is not in $L^1_{\rm loc}(\Omega)$ for all small $t>0$. It is an interesting open question whether such strong blowup still occurs under the weaker condition in Theorem \ref{T1}.

A combination of the blowup result of Theorem \ref{T1} and classical results for the Fujita equation now give our first characterisation theorem, on local existence in $L^q(\Omega)$ when $q>1$.

\begin{theorem}\label{thm:char1}
 Let $f:[0,\infty)\to[0,\infty)$ be non-decreasing and continuous. If $q\in(1,\infty)$ then (\ref{eqn}) has the local existence property in $L^q(\Omega)$ if and only if
\begin{equation}\label{finite}
\limsup_{s\to\infty}s^{-(1+2q/d)}f(s)<\infty.
\end{equation}
\end{theorem}

\begin{proof}
  It remains only to show that (\ref{eqn}) has a local solution bounded in $L^q(\Omega)$ when (\ref{finite}) holds. In this case it follows that
there exists a constant $C$ such that
$$
  f(s)\le C(1+s^p),\quad\mbox{where}\quad p=1+2q/d,
$$
  and now one can use comparison (Theorem 1 in Robinson \& Sier{\.z}\polhk{e}ga \cite{RS2}) and standard existence results for the equation
$$
  u_t-\Delta u=C(1+u^p)
$$
  (Corollary 3.2 in Weissler \cite{W80}) to guarantee that (\ref{eqn}) has the local $L^q(\Omega)$ existence property.
\end{proof}

One could rephrase the above result in terms of the quantity
$$
\gamma^\star=\sup\{\gamma\ge0:\ \limsup_{s\to\infty}s^{-\gamma}f(s)=\infty\}.
$$
With $q^\star=d(\gamma^\star-1)/2$ equation (\ref{eqn}) does not enjoy local existence for all non-negative initial data in $L^q$ for $q<q^\star$, but does for $q>q^\star$. In this way $q^\star$ defines a `critical exponent' for the general class of non-decreasing $f$ we consider here. Provided that $q^\star>1$ local existence/non-existence in the critical space $L^{q^\star}$ is determined by the behaviour of
$$
\limsup_{s\to\infty}s^{-\gamma^\star}f(s).
$$
When $q^\star=1$ the situation is more delicate and somewhat surprising.


\section{Initial data in $L^1(\Omega)$}

\subsection{A condition for non-existence of a local $L^1$ solution}\label{noL1}

  As just remarked, the behaviour of solutions for initial data in $L^1(\Omega)$ is more delicate. Celik \& Zhou \cite{CZ} showed that when $f(s)=s^{1+2/d}$, there is initial data in $L^1(\Omega)$ for which there is no local $L^1$ solution. This suggests that when $q=1$ the requirement of Theorem \ref{T1} can be weakened. Indeed, the requirement that the sum in (\ref{ellkinf}) diverges is clearly weaker than the asymptotic condition,
$$
  \limsup_{s\to\infty}s^{-(1+2/d)}f(s)>0;
$$
  blowup can even occur for certain $f$ for which the above $\limsup$ is zero, such as  $f(s)=s^{1+2/d}/\log(\e+s)^\beta$ with $0<\beta\le1$.
In particular, algebraic growth $f(s)=s^{1+2/d}$ is not in fact the true `boundary' for $L^1$ blowup. We examine this example in a little more detail in Section \ref{logex}.

Note that in the statement of the following theorem we do not include the hypothesis that $f$ is continuous; this is not required for this blowup result.

\begin{theorem}\label{Todd}
Suppose that $f:[0,\infty)\to[0,\infty)$ is non-decreasing and that there exists a sequence $\{s_k\}$ such that
$$
s_{k+1}\ge\theta s_k,\qquad\theta>1,
$$
and
\begin{equation}\label{ellkinf}
\sum_{k=1}^\infty s_k^{-p}f(s_k)=\infty,
\end{equation}
where $p=1+\frac{2}{d}$. Then there exists a non-negative initial condition $u_0\in L^1(\Omega)$ such that
\begin{equation}\label{ours}
u_t-\Delta u=f(u),\qquad u|_{\partial\Omega}=0,\qquad u(x,0)=u_0
\end{equation}
has no local integral solution that remains in $L^1_{\rm loc}(\Omega)$ for $t>0$ (so in particular no local $L^1$ solution exists).
\end{theorem}

Before we give the proof proper, it is instructive to present a much simplified argument for the `$L^1$-like' initial data $u_0=\delta_0$, a delta function centred at the origin (such data is $L^1$-like so far as $\int \delta_0=1$). To further simplify the argument we pose the problem on the whole space $\R^d$.

Since in this case
$$
[S(s)\delta_0](x)=(4\pi s)^{-d/2}\e^{-|x|^2/4s}
$$
it follows that for each $k$,
$$
S(s)\delta_0(x)\ge\phi_k\chi_{\sqrt s}\qquad\mbox{for}\quad s\le t_k:=c\phi_k^{-2/d}
$$
(where $c=\e^{-1/2d}/4\pi$).

Now for any $t>0$, using the fact that $\int_{\R^d} S(t)\chi_r=\omega_dr^d$,
\begin{align*}
\int_{\R^d}\int_0^t S(t-s)f(S(s)\delta_0)\,\d s
&\ge \int_{\R^d}\sum_k\int_{t_{k+1}}^{t_k} S(t-s)f(S(s)\delta_0)\,\d s\\
&\ge \sum_kf(\phi_k)\int_{t_{k+1}}^{t_k}\int_{\R^d} S(t-s)\chi_{\sqrt s}\,\d s\\
&=\omega_d\sum_kf(\phi_k)\int_{t_{k+1}}^{t_k} s^{d/2}\,\d s\\
&=c\omega_d\sum_kf(\phi_k) (t_k^{(2+d)/2}-t_{k+1}^{(2+d)/2})\\
(t_k=c\phi_k^{-2/d})\qquad\qquad&\ge c'\omega_d\sum_kf(\phi_k) (\phi_k^{-p}-\phi_{k+1}^{-p})\\
(\phi_{k+1}\ge\theta\phi_k)\qquad\qquad&\ge c'(1-\theta^{-p})\omega_d\sum_k f(\phi_k)\phi_k^{-p}=\infty,
\end{align*}
using (\ref{ellkinf}). The proof of Theorem \ref{Todd} will follow very similar lines.

\begin{proof} (Theorem \ref{Todd}.)
Define $\phi_k=c_d^{-1}s_k$, and set
$$
u_n(x)=\frac{1}{n^2}\alpha_n^d\chi(1/\alpha_n)\qquad\mbox{where}\qquad
\alpha_n=(n^2\phi_{\zeta_n})^{1/d},
$$
with $\zeta_n$ to be chosen later. Let
$$
u_0(x)=\sum_{n=n_0}^\infty u_n(x),
$$
with $n_0$ chosen such that
$$
\frac{1}{\alpha_{n_0}}<\delta_0:=\frac{1}{2}\inf_{x\in\partial\Omega}|x|.
$$
Note that $1/\alpha_n\le\delta_0$ and so $B_{1/\alpha_n+\delta_0}\subset\Omega$ for all $n\ge n_0$, and that
$$
\|u_0\|_{L^1}\le\omega_d\sum_{n=1}^\infty n^{-2}<\infty.
$$
Arguing as in the proof of Theorem \ref{T1}, for any choice of $n$ we have
$$
\int_\Omega u(t;u_0)\,\d x\ge \int_\Omega\int_0^t S(t-s)f(S(s)u_n)\,\d s\,\d x.
$$

We now consider the action of the heat semigroup on the initial data $v_0=\psi\alpha^d\chi_{1/\alpha}$. It follows from Lemma \ref{cflb} with $r=1/\alpha$ and $\delta=\delta_0$ that
$$
S(s)v_0\ge c_d\psi\frac{\alpha^d}{(1+\alpha^2 s)^{d/2}}\chi_{(1/\alpha)+\sqrt s},
$$
so $[S(s)v_0](x)\ge c_d\phi_k$ for
$$
|x|\le \frac{1}{\alpha}+\sqrt s\quad\mbox{while}\quad s\le t_k=\min\left(\delta_0^2,\left(\frac{\psi}{\phi_k}\right)^{2/d}-\frac{1}{\alpha^2}\right),
$$
and this range is non-empty provided that $\phi_k\le\psi\alpha^d$.

Now for any $0<t<\delta_0^2$, using Corollary \ref{norm1} we have
\begin{align*}
\int_\Omega\int_0^t S(t-s)f(S(s)v_0)&\,\d s\,\d x\ge \sum_k\int_\Omega\int_{t_{k+1}}^{t_k} S(t-s)f(S(s)v_0)\,\d s\,\d x\\
&=\sum_k\int_{t_{k+1}}^{t_k}\int_\Omega S(t-s)f(S(s)v_0)\,\d x\,\d s\\
&\ge c'\sum_kf(c_d\phi_k)\int_{t_{k+1}}^{t_k}\int_\Omega S(t-s)\chi_{\alpha^{-1}+\sqrt s}\,\d x\,\d s\\
&\ge \alpha_d c'\sum_kf(c_d\phi_k)\int_{t_{k+1}}^{t_k}\left(\frac{1}{\alpha}+\sqrt s\right)^d\,\d s\\
&\ge c''\sum_kf(c_d\phi_k)\int_{t_{k+1}}^{t_k} s^{d/2}\,\d s,
\end{align*}
where the sum in $k$ is taken over those values for which
$$
\frac{1}{\alpha^d}\le\frac{\psi}{\phi_k}\le \left(t+\frac{1}{\alpha^2}\right)^{d/2}.
$$

Let us consider $k$ that satisfy this requirement and the additional constraint that $\phi_{k+1}/\alpha^d\psi\le 1/2$. For each such $k$ we have
\begin{align*}
\int_{t_{k+1}}^{t_k}& s^{d/2}\,\d s
=\frac{2}{2+d}(t_k^{d/2+1}-t_{k+1}^{d/2+1})\\
&=\frac{2}{2+d}\left\{\left[\left(\frac{\psi}{\phi_k}\right)^{2/d}-\frac{1}{\alpha^2}\right]^{d/2+1}
-\left[\left(\frac{\psi}{\phi_{k+1}}\right)^{2/d}-\frac{1}{\alpha^2}\right]^{d/2+1}\right\}\\
&=\frac{2}{2+d}\left(\frac{\psi}{\phi_k}\right)^{1+2/d}\left[\left(1-\left(\frac{\phi_k}{\alpha^d\psi}\right)^{2/d}\right)^{d/2+1}\right.\\
&\qquad\qquad\qquad\qquad\qquad\qquad\left.
-\frac{\phi_k^{1+2/d}}{\phi_{k+1}^{1+2/d}}\left(1-\left(\frac{\phi_{k+1}}{\alpha^d\psi}\right)^{2/d}\right)^{d/2+1}\right]\\
&\ge\frac{2}{2+d}\left(\frac{\psi}{\phi_k}\right)^{1+2/d}\left(1-\frac{\phi_k^{1+2/d}}{\phi_{k+1}^{1+2/d}}\right)
\left(1-\left(\frac{\phi_{k+1}}{\alpha^d\psi}\right)^{2/d}\right)^{d/2+1}\\
&\ge \sigma\left(\frac{\psi}{\phi_k}\right)^{1+2/d},
\end{align*}
using the facts that $\phi_{k+1}\ge\theta\phi_k$ and $\phi_{k+1}/\alpha^d\psi\le 1/2$. So certainly
\begin{align*}
\int_\Omega\int_0^t S(t-s)f(S(s)v_0)\,\d s&\ge c''\sigma\sum_kf(c_d\phi_k)\left(\frac{\psi}{\phi_k}\right)^{1+2/d}\\
&=c'''\psi^p\sum_k f(s_k)s_k^{-p},
\end{align*}
where the sum is taken over
\begin{equation}\label{tofindk}
\left\{k:\ \frac{2}{\alpha^d}\le\frac{\psi}{\phi_{k+1}}<\frac{\psi}{\phi_k}\le \left(t+\frac{1}{\alpha^2}\right)^{d/2}\right\}.
\end{equation}

For any fixed $t$ with $0<t<\delta_0^2$, once $n$ is sufficiently large that $tn^{4/d}\ge 1$ the set in (\ref{tofindk}) with $\psi=n^{-2}$ and $\alpha=\alpha_n=(n^2\phi_{\zeta_n})^{1/d}$ certainly contains
$$
\{k:\ 1\le\phi_k\qquad\mbox{and}\qquad \phi_{k+1}\le {\textstyle\frac{1}{2}}\phi_{\zeta_n}\}=\{k:\ k_0\le k\le k_n\},
$$
where $k_0$ is the smallest value of $k$ for which $\phi_k\ge 1$ and by choosing $\zeta_n$ such that $\phi_{k_n+1}\le\frac{1}{2}\phi_{\zeta_n}$ we can achieve any desired sequence $k_n$.

Since $\sum_{k=1}^\infty f(s_k)s_k^{-p}=\infty$ (by (\ref{ellkinf})) we can choose $k_n$ such that
$$
n^{-2p}\sum_{k=k_0}^{k_n}f(s_k)s_k^{-p}
$$
diverges as $n\to\infty$.\end{proof}

We note that if we assume in addition
that $f(s)\ge cs$ for some $c>0$, then under the conditions in Theorem \ref{Todd} there is in fact no local integral solution of (\ref{ours}). Indeed, suppose that there is a local integral solution $u:\Omega\times[0,T)\to[0,\infty)$. Then by Definition \ref{lis}, $u$ is finite almost everywhere on $\Omega\times[0,T)$. Since all our estimates are performed within $B_{\delta_0}$, we have in fact shown that there is a ball $B\subset\Omega$ and a time $\delta_0^2>0$ such that
$$
\int_B u(y,s)\,\d y=\infty\qquad\mbox{for all}\quad s\in(0,\delta_0^2).
$$
Now fix $\tau=\min(T,\delta_0^2)$ and choose any $(x,t)\in B\times[\tau/2,\tau]$. Since $u$ satisfies (\ref{eq:mild}),
\begin{align*}
u(x,t)&\ge\int_0^t\int_\Omega K(x,y;t-s)f(u(y,s))\,\d y\,\d s\\
&\ge c\int_0^{\tau/4}\int_B K(x,y;t-s)u(y,s)\,\d y\,\d s,
\end{align*}
using the assumption that $f(s)\ge cs$. For $s\in[0,\tau/4]$ and $t\in[\tau/2,\tau]$ we have $t-s\in[\tau/4,\tau]$. By the continuity and positivity of $K$ there exists $\kappa>0$ such that
$$
K(x,y;\sigma)\ge\kappa\qquad\mbox{for all}\quad (x,y,\sigma)\in B\times B\times [\tau/4,\tau],
$$
whence
$$
u(x,t)\ge c\kappa\int_0^{\tau/4}\int_B u(y,s)\,\d y\,\d s=\infty.
$$
Therefore $u=\infty$ on $B\times[\tau/2,\tau]$, contradicting the requirement that $u$ is finite almost everywhere on $\Omega\times[0,T)$.

\subsection{An equivalent integral condition for blowup}

Since the condition in (\ref{ellkinf}) is potentially awkward to check in practice, we now formulate an equivalent integral condition. Note that when $f(s)/s$ is non-decreasing, the integral condition in (ii) of the lemma below becomes the more conventional
$$
\int_1^\infty s^{-(1+p)}f(s)\,\d s=\infty.
$$

\begin{lemma}\label{equivcond}
Suppose that $f:[0,\infty)\to[0,\infty)$ is non-decreasing and $p>1$. Then the following two conditions are equivalent.
\begin{itemize}
\item[(i)] There exists a sequence $\{s_k\}$ such that $s_{k+1}\ge\theta s_k$, $\theta>1$ and
$$
    \sum_{k=1}^\infty s_k^{-p}f(s_k)=\infty.
$$
\item[(ii)] $\displaystyle{
\int_1^\infty s^{-p}F(s)\,\d s=\infty,\qquad\mbox{where}\qquad F(s)=\sup_{1\le t\le s}\frac{f(t)}{t}.
}$
\end{itemize}
\end{lemma}

\begin{proof}
First we show that (i) implies (ii). We can augment the sequence $\{s_k\}$ to a new sequence $\sigma_k$ such that
$$
\sum_{k=0}^\infty\sigma_k^{-p}f(\sigma_k)=\infty
$$
and in addition, choosing $1<\alpha<p$,
$$
1<\theta\le\frac{\sigma_{k+1}}{\sigma_k}\le\theta^\alpha,
$$
by including points $\theta^js_k$ until $s_{k+1}\le\theta^{j+p-1}s_k$.

Setting $\sigma_0=1$ we can write
\begin{align*}
\int_1^\infty s^{-p}F(s)\,\d s&=\sum_{k=0}^\infty\int_{\sigma_k}^{\sigma_{k+1}} s^{-p}F(s)\\
&\ge\sum_{k=0}^\infty\int_{\sigma_k}^{\sigma_{k+1}}s^{-p}F(\sigma_k)\,\d s\\
&\ge\frac{1}{p-1}\sum_{k=0}^\infty (\sigma_k^{-(p-1)}-\sigma_{k+1}^{-(p-1)})\frac{f(\sigma_k)}{\sigma_{k+1}}\\
&=\frac{1}{p-1}\sum_{k=0}^\infty \sigma_k^{-p}f(\sigma_k)\left\{\frac{\sigma_k}{\sigma_{k+1}}-\frac{\sigma_k^p}{\sigma_{k+1}^p}\right\}\\
&\ge\frac{1}{p-1}(\theta^{-\alpha}-\theta^{-p})\sum_{k=0}^\infty\sigma_k^{-p}f(\sigma_k),
\end{align*}
from which (ii) follows.

We now show that (ii) implies (i). Choose $\theta>1$ and for $k=0,1,2,\ldots$ let $\sigma_k=\theta^k$; note that $F(s)\le F(\sigma_{n+1})$ for all $s\in(\sigma_n,\sigma_{n+1}]$. There exists a sequence $\{k_n\}$ with $k_n\le n$ and $k_{n+1}\ge k_n$ such that $F(\sigma_{n+1})=f(\tau_n)/\tau_n$ for some $\tau_n\in(\sigma_{k_n},\sigma_{k_n+1}]$. Thus
$$
F(s)\le F(\sigma_{n+1})\le\frac{f(\sigma_{k_n+1})}{\sigma_{k_n}}\qquad\mbox{for all}\quad s\in(\sigma_n,\sigma_{n+1}].
$$
Therefore
$$
\int_1^\infty s^{-p}F(s)\,\d s=\sum_{n=1}^\infty \int_{\sigma_n}^{\sigma_{n+1}} s^{-p}F(s)\,\d s\le\sum_{n=1}^\infty \frac{f(\sigma_{k_n+1})}{\sigma_{k_n}}\int_{\sigma_n}^{\sigma_{n+1}}s^{-p}\,\d s.
$$
%
$$
%
$$

Now observe that there is an increasing sequence $n_j$ such that
$$
k_{n_j}=k_n<k_{n_{j+1}}\qquad\mbox{for}\quad n=n_j,\ldots,n_{j+1}-1,
$$
and so
\begin{align*}
\sum_{n=1}^\infty \frac{f(\sigma_{k_n+1})}{\sigma_{k_n}}\int_{\sigma_n}^{\sigma_{n+1}}s^{-p}\,\d s&=\sum_j \frac{f(\sigma_{k_{n_j}+1})}{\sigma_{k_{n_j}}}\int_{\sigma_{k_{n_j}}}^{\sigma_{k_{n_{j+1}}}}s^{-p}\,\d s\\
&<\sum_j\frac{f(\sigma_{k_{n_j}+1})}{\sigma_{k_{n_j}}}\int_{\sigma_{k_{n_j}}}^\infty s^{-p}\,\d s\\
&\le\frac{1}{p-1}\sum_{n=1}^\infty \frac{f(\sigma_{k_{n_j}+1})}{\sigma_{k_{n_j}}}\,\sigma_{k_{n_j}}^{1-p}\\
&=\frac{\theta^p}{p-1}\sum_{j=1}^\infty\sigma_{k_{n_j}+1}^{-p}f(\sigma_{k_{n_j}+1}).
\end{align*}
Taking $s_j=\sigma_{k_{n_j}+1}$ yields (i).\end{proof}

\subsection{An integral condition for local existence}

We now show that the integral condition in (ii) of Lemma \ref{equivcond} is sufficient for the $L^1$ local existence property. We will use the following theorem from Robinson \& Sier{\.z}\polhk{e}ga \cite{RS2} (Theorem 1, after Weissler \cite{W81}) which guarantees the existence of a solution $u(t)$ of (\ref{nhe}) given the existence of a supersolution $v(t)$, i.e.\ a function satisfying (\ref{supersol}). For later use we remark that $\Omega=\R^d$, with $S(t)$ denoting the action of the heat semigroup (defined by convolution with the Gaussian kernel) is an admissible choice in Theorem \ref{MS} (see discussion in the `Final comments' in \cite{RS2}).

  \begin{theorem}\label{MS}
   Take $u_0\ge0$. If $f:[0,\infty)\to[0,\infty)$ is continuous and non-decreasing and there exists a $v\in L^1((0,T)\times\Omega)$ such that
  \begin{equation}\label{supersol}
  S(t)u_0+\int_0^t S(t-s)f(v(s))\,\d s\le v(t)\qquad\mbox{for all}\quad t\in[0,T]
  \end{equation}
  then there exists a local integral solution $u$ of (\ref{nhe}) on $[0,T]$ with $u(x,t)\le v(x,t)$ for all $x\in\Omega$ and $t\in[0,T]$.
  \end{theorem}

  This theorem is proved by constructing a sequence of supersolutions $v_n(t)$ defined by setting $v_0(t)=v(t)$ and
$$
  v_{n+1}(t)={\mathscr F}(v_n):=S(t)u_0+\int_0^t S(t-s)f(v_n(s))\,\d s.
$$
  Such a sequence is monotonically decreasing, is bounded below by $S(t)u_0$, and hence has a pointwise limit $u(t)$ which can be shown to satisfy
$$
   S(t)u_0+\int_0^t S(t-s)f(u(s))\,\d s=u(t)\qquad\mbox{for all}\quad t\in[0,T]
$$
 using the Monotone Convergence Theorem.

  Using this result we prove a local existence theorem; the argument is adapted from the proof of Proposition 7.2 in Sier{\.z}\polhk{e}ga \cite{MT}. Note that our standing assumption that $\Omega$ is bounded is an important ingredient in the proof, since we require $\chi_\Omega\in L^1(\Omega)$.

\begin{theorem}\label{T4}
If $f:[0,\infty)\to[0,\infty)$ is continuous, non-decreasing, and
 \begin{equation}\label{Mikos2}
  \int_1^\infty s^{-(1+2/d)}F(s)\,\d s<\infty,\qquad\mbox{where}\qquad F(s)=\sup_{1\le t\le s}\frac{f(t)}{t}
  \end{equation}
  then for every non-negative $u_0\in L^1(\Omega)$ there exist a $T>0$ such that (\ref{eqn}) has a solution
$$
  u\in L^\infty_{\rm loc}((0,T);L^\infty(\Omega))\cap C^0([0,T];L^1(\Omega)).
$$
  In particular, (\ref{eqn}) has the local $L^1$ existence property.
\end{theorem}

\begin{proof}
If $u_0=0$ then $v(t)\equiv\chi_\Omega$ is a supersolution, since
$$
S(t)u_0+\int_0^t S(t-s)f(S(s)\chi_\Omega)\,\d s\le\int_0^t S(t-s)\{f(1)\chi_\Omega\}\le tf(1)\chi_\Omega\le\chi_\Omega
$$
for all $t$ sufficiently small.

To treat $u_0\neq0$, define $\tilde f(s)=f(s)$ for $s\in[0,1]$ and $\tilde f(s)=sF(s)$ for $s>1$. Then $f(s)\le\tilde f(s)$ and $\tilde f(s)/s:[1,\infty)\to[0,\infty)$ is non-decreasing. In particular, any supersolution for the equation
\begin{equation}\label{tildeeqn}
u_t-\Delta u=\tilde f(u)
\end{equation}
is also a supersolution for (\ref{eqn}), and therefore to show that (\ref{eqn}) has a solution it suffices to find a supersolution for (\ref{tildeeqn}).

Rewritten in terms of $\tilde f$, the integral condition in (\ref{Mikos2}) becomes
$$
\int_1^\infty s^{-(2+2/d)}\tilde f(s)\,\d s<\infty,
$$
and after the substitution $s=\tau^{-d/2}$ we obtain
$$
\int_0^1 \tau^{d/2}\tilde f(\tau^{-d/2})\,\d\tau<\infty.
$$

We now show that for any $A>1$, $v(t)=AS(t)u_0+\chi_\Omega$ is a supersolution of (\ref{tildeeqn}) on some suitable time interval, i.e.\ satisfies the condition (\ref{supersol}) in Theorem \ref{MS}. In order to do this, first recall the smoothing estimate
$$
\|S(t)u_0\|_{L^\infty}\le ct^{-d/2}\|u_0\|_{L^1}.
$$

We therefore obtain
\begin{align*}
S(t)u_0+\int_0^t& S(t-s)\tilde f(v(s))\,\d s=S(t)u_0+\int_0^t S(t-s)f(AS(s)u_0+1)\,\d s\\
&=S(t)u_0+\int_0^t S(t-s)\left(\frac{\tilde f(AS(s)u_0+1)}{AS(s)u_0+1}\right)(AS(s)u_0+1)\,\d s\\
&\le S(t)u_0+\int_0^t S(t-s)\left\|\frac{\tilde f(AS(s)u_0+1)}{AS(s)u_0+1}\right\|_{L^\infty}(AS(s)u_0+1)\,\d s.
\end{align*}
Since the $L^\infty$ norm is a scalar constant and $S(t-s)$ is linear, it follows that
$$
{\mathscr F}(v)(t)\le S(t)u_0+\left\{\int_0^t\left\|\frac{\tilde f(AS(s)u_0+1)}{AS(s)u_0+1}\right\|_{L^\infty}\,\d s\right\}[S(t)u_0+\chi_\Omega],
$$
as $S(t)\chi_\Omega\le\chi_\Omega$ for all $t>0$. 
Now, using the fact that $\tilde f(s)/s$ is non-decreasing for $s\ge1$,
\begin{align*}
{\mathscr F}(v)(t)&\le
S(t)u_0+\left\{\int_0^t\frac{\tilde f(\|AS(s)u_0+1\|_{L^\infty})}{\|AS(s)u_0+1\|_{L^\infty}}\,\d s\right\}[AS(t)u_0+\chi_\Omega]\\
&\le
S(t)u_0+\left\{\int_0^t\frac{\tilde f(2Acs^{-d/2}\|u_0\|_{L^1})}{2Acs^{-d/2}\|u_0\|_{L^1}}\,\d s\right\}[AS(t)u_0+\chi_\Omega],
\end{align*}
for $s$ sufficiently small, since
$$
\|AS(s)u_0+1\|_{L^\infty}=\|AS(s)u_0\|_{L^\infty}+1\le Acs^{-d/2}\|u_0\|_{L^1}+1\le 2Acs^{-d/2}\|u_0\|_{L^1}
$$
for $s$ sufficiently small.

Therefore ${\mathscr F}(v)(t)$ is bounded above by
\begin{align*}
S(t)u_0+&(2Ac\|u_0\|_{L^1})^{2/d}\left(\int_0^{t(2Ac\|u_0\|_{L^1})^{-2/d}}\tau^{d/2}\tilde f(\tau^{-d/2})\,\d\tau\right)[AS(t)u_0+\chi_\Omega]\\
&\le AS(t)u_0+\chi_\Omega,
\end{align*}
provided that $t$ is sufficiently small. Local existence of a solution $u(t)$ with $u(t)\le v(t)=2S(t)u_0+\chi_\Omega$ now follows from Theorem \ref{MS}. That $u(t)$ is bounded in $L^1(\Omega)$ now follows from Theorem \ref{MS}. It follows that $u\in L^\infty_{\rm loc}((0,T);L^\infty(\Omega))$, i.e.\ is a classical solution. Then from the integral condition on $f$, $f(u)\in L^1((0,T);L^1(\Omega))$, whence from (\ref{eq:mild}) it follows that $u\in C^0([0,T];L^1(\Omega))$.
\end{proof}

We have therefore obtained the following characterisation of those $f$ for which there is local existence in $L^1(\Omega)$.

\begin{corollary}\label{corL1}
If $f:[0,\infty)\to[0,\infty)$ is continuous and non-decreasing then (\ref{eqn}) has the local $L^1$ existence property if and only if
$$
  \int_1^\infty s^{-(1+2/d)}F(s)\,\d s<\infty,\qquad\mbox{where}\qquad F(s)=\sup_{1\le t\le s}\frac{f(t)}{t}.
$$
\end{corollary}

We note that one can apply the `local existence' part of this characterisation (i.e.\ Theorem \ref{T4}) to a nonlinearity $g$ that is not non-decreasing by finding a non-decreasing function $f(s)$ such that $g(s)\le f(s)$, applying Theorem \ref{T4} and then deducing local existence by comparison. The example of the following section provides an example of this along with an illustration of the application of Corollary \ref{corL1}.

\subsection{An example: $f(s)=s^{1+2/d}/[\log(\e+s)]^\beta$}\label{logex}

We mentioned before the proof of Theorem \ref{Todd} that the family of nonlinearities
$$
f(s)=\frac{s^p}{[\log(\e+s)]^\beta},\qquad p=1+2/d,\qquad\beta\ge0,
$$
provides an interesting set of examples, particularly in the light of the (erroneous) expectation that $f(s)=s^{1+2/d}$ lies on the `boundary' between those functions for which (\ref{nhe}) does and does not have the local $L^1$ existence property.

Strictly, such a function $f$ only falls within the scope of our results when it is non-decreasing, which occurs if and only if $\beta\le\lambda p$, where $\lambda\simeq 3.15$ is the largest positive root of the equation $\e^x=\e^2x$. It follows from Corollary \ref{corL1} that
\begin{itemize}
\item[(i)] if $0\le\beta\le 1$ then (\ref{nhe}) does not have the $L^1$ local existence property;
\item[(ii)] if $1<\beta\le\lambda p$ then (\ref{nhe}) does have the $L^1$ local existence property;
\end{itemize}
and since although when $\beta>\lambda p$ the function $f(s)$ is not monotone, it is bounded above by the monotone $s^p/\log(\e+s)^{\lambda p}$, which provides a supersolution and hence
\begin{itemize}
\item[(iii)] if $\beta>\lambda p$ then (\ref{nhe}) does have the $L^1$ local existence property;
\end{itemize}

Within this family the function $f(s)=s^p/\log(\e+s)$ lies on the `boundary'. Obviously one could refine this with the addition of an arbitrary number of repeated logarithms.

\section{Results for the whole space and for Neumann boundary conditions}

It is worth remarking that since they rely only on Gaussian lower bounds for the Dirichlet heat kernel, the non-existence results of Theorems \ref{T1} and \ref{Todd} are valid with essentially the same proofs for the equations posed on the whole space $\R^d$. They are also valid for Neumann boundary conditions ($\frac{\partial u}{\partial n}=0$ on $\partial\Omega$), since
$$
K_{\Omega_N}(x,y;t)\ge K_{\Omega_D}(x,y;t),\qquad x,y\in\Omega,\ t>0,
$$
where $\Omega_N$ and $\Omega_D$ denote the Neumann and Dirichlet heat kernels, respectively (the proof follows by comparison, or one can use probabilistic methods, see Corollary 2.5 in \cite{vdB89}, for example).

However, local existence results on the whole space require some additional assumptions. It is easy to see that if $f(0)\neq0$ then any non-negative initial condition gives rise to a solution that is not in $L^q(\R^d)$ for any $q\in[1,\infty)$. Indeed, since then $u(t)\ge0$ for all $t\ge0$ we have
$$
u(t)\ge \int_0^t S(t-s)f(u(s))\,\d s\ge\int_0^t S(t-s)f(0)\,\d s=tf(0)\notin L^q(\R^d).
$$
We also require a `bounded derivative at zero' condition, namely
\begin{equation}\label{limsup0}
\limsup_{s\to0}\frac{f(s)}{s}<\infty.
\end{equation}
Without this condition we can find a non-negative $u_0\in L^q(\R^d)$ for which the solution is not bounded in $L^q(\R^d)$ for all $t>0$.

  Indeed, if (\ref{limsup0}) does not hold then there exist $s_n\to0$ such that $s_n\le n^{-2}$ and $f(s_n)\ge n^2s_n$. Consider initial data
$$
u_0=\sum_{n=1}^\infty c_d^{-1}s_n\chi_{n^{-2/d}s_n^{-q/d}}(x_n)
$$
where the $x_n$ are chosen such that $B(x_n,n^{-2/d}s_n^{-q/d})$ are disjoint. Note that $\|u_0\|_{L^q}<\infty$ and that $n^{-2/d}s_n^{-q/d}\ge 1$.

Then
$$
S(s)u_0\ge \sum_{n=n_0}^\infty s_n\chi_{n^{-2/d}s_n^{-q/d}}(x_n)
$$
for all $s\le1$. So for $t\le 1$ we have
\begin{align*}
u(x,t)&\ge\int_0^t\sum_{n=1}^\infty S(t-s) n^2  s_n \chi_{n^{-2/d}s_n^{-q/d}}(x_n)\,\d s\\
&\ge\int_0^t\sum_{n=1}^\infty c_d n^2 s_n \chi_{n^{-2/d}s_n^{-q/d}}(x_n)\,\d s\\
&= tc_d\sum_{n=1}^\infty n^2 s_n \chi_{n^{-2/d}s_n^{-q/d}}(x_n)
\end{align*}
and so
$$
\int_{\R^d} |u(x,t)|^q\,\d x\ge (tc_d)^q\sum_{n=1}^\infty n^{2q}s_n^qn^{-2}s_n^{-q}=tc_d\sum_{n=1}^\infty n^{2(q-1)}=\infty.
$$

Now, for $q>1$ if we have, with $p=1+2q/d$,
$$
\limsup_{s\to\infty}s^{-p}f(s)<\infty\qquad\mbox{and}\qquad\limsup_{s\to0}\frac{f(s)}{s}<\infty
$$
then
$$
f(s)\le C(s+s^p)
$$
for some $C>0$. For $f(s)=C(s+s^p)$ we can guarantee the local $L^q$ existence property on $\R^d$ as follows. First, results guaranteeing the $L^q$ local existence property on the whole space when $f(s)=2Cs^p$ can be found in Weissler \cite{W79,W80} (the analysis there is valid on the whole space), Theorem 1 in Giga \cite{G}, or Robinson \& Sier{\.z}{\polhk{e}}ga \cite{RS2} (see `Final comments'). So given a non-negative $u_0\in L^q(\R^d)$, let $u(t)$ be the local $L^q$ solution obtained in this way. Now define $v(t)=\e^{2Ct}u(t)$. Then
\begin{align*}
v_t-\Delta v&=2C\e^{2Ct}u+\e^{2Ct}(u_t-\Delta u)\\
&=2C\e^{2Ct}u+\e^{2Ct}2Cu^p\\
&=2C(v+(\e^{2Ct})^{1-p}v^p)\\
&\ge C(v+v^p)
\end{align*}
provided that $(\e^{2Ct})^{1-p}\ge 1/2$. This is legitimate  since $v$ is a strong, classical solution for $t>0$. It then follows easily that $v$ is a supersolution in the integral sense of Theorem \ref{MS} on some small time interval. The existence of such a supersolution, which is bounded in $L^q(\R^d)$, then implies the existence of a solution bounded in $L^q(\R^d)$ using Theorem \ref{MS}.

For $q=1$ local existence follows from the arguments in Theorem \ref{T4}, now taking
$$
F(s)=\sup_{0\le t\le s}\frac{f(t)}{t}
$$
and observing that if $u_0\ge0$ and is non-zero then $S(t)u_0>0$ for all $t>0$ (this follows immediately from the expression for $S(t)u_0$ in terms of the Gaussian kernel). (If $u_0=0$ then obviously $u(t)=0$ is a solution since $f(0)=0$ by the limsup condition at $s=0$.)

We summarise formally in the following theorem.

\begin{theorem}\label{Rdresult}
 Let $f:[0,\infty)\to[0,\infty)$ be continuous and non-decreasing and let $\Omega=\R^d$. Then
 \begin{itemize}
 \item[(i)] for $q\in(1,\infty)$ equation (\ref{nhe}) has the local $L^q$ existence property if and only if
$$
\limsup_{s\to0}\frac{f(s)}{s}<\infty\quad\mbox{and}\quad\limsup_{s\to\infty}s^{-(1+2q/d)}f(s)<\infty;
$$
\item[(ii)] equation (\ref{nhe}) has the local $L^1$ existence property if and only if
$$
\limsup_{s\to0}\frac{f(s)}{s}<\infty\quad\mbox{and}\quad\int_1^\infty s^{-(1+2/d)}F(s)\,\d s<\infty,
$$
where $F(s)=\sup_{0\le t\le s}\frac{f(t)}{t}$.
\end{itemize}
\end{theorem}

\section{Concluding remarks}

We have completely characterised those non-negative, non-decreasing, continuous functions $f$ for which the equation
$$
u_t-\Delta u=f(u),\qquad u|_{\partial\Omega}=0,\qquad u(0)=u_0\in L^q(\Omega),\ u_0\ge0,
$$
has at least one local solution that is bounded in $L^q(\Omega)$. For $1<q<\infty$ this occurs if and only if
$$
\limsup_{s\to\infty}s^{-(1+2q/d)}f(s)<\infty,
$$
while for $q=1$ this occurs if and only if
$$
 \int_1^\infty s^{-(1+2/d)}\left(\sup_{1\le t\le s}\frac{f(t)}{t}\right)\,\d s<\infty.
$$
 We have also given results for the equations on the whole space $\R^d$ and for the Neumann problem on a bounded domain.

The non-existence parts of our arguments are perhaps the most novel, using lower bounds on the Dirichlet heat kernel due to van den Berg \cite{vdB89,vdB1} to give lower bounds on solutions of the heat equation with characteristic functions as initial data, and hence lower bounds on solutions of the semilinear problem. The $L^1$ case behaves very differently from the problem in spaces with higher integrability, with the appearance of the upper bound
$$
sF(s)=s\sup_{1\le t\le s} [f(t)/t]
$$
in both the blowup and existence criteria something of a surprise.

An open question left by our $L^q$-based analysis is whether we have `strong blowup', $u(t)\notin L^q$ for \emph{all} $t>0$, when $q>1$ and
$$
\limsup_{s\to\infty}s^{-\gamma}f(s)=\infty,\qquad \gamma\in[1+2q/d,q(1+2/d)],
$$
or whether there is a true transition from such strong blowup (obtained in \cite{LRS1} for $\gamma>q(1+2/d)$ to only the unbounded behaviour $\limsup_{t\to0}\|u(t)\|_{L^q}$ (obtained here in Theorem \ref{T1} when $\gamma=1+2q/d$). A related question is whether it is possible to exclude the existence of local integral solutions for a wider class of $f$ than we do in the discussion at the end of Section \ref{noL1} (we currently require $f(s)\ge Cs$ for some $C>0$).

It would be interesting to attempt to prove similar characterisation results in other scales of spaces, such as Sobolev spaces or Besov spaces. These would require different techniques, given that our current arguments do not take into account the smoothness of solutions but only their integrability.

 Seeking generalisation in a different direction, one could ask whether there is a way of identifying the critical Lebesgue space for the more general class of positive but not necessarily monotone $f$, or even for general $f$ with sign-changing initial data.

Finally, we note that we have not attempted here to consider the problem of uniqueness. For the nonlinearity $f(u)=|u|^{p-1}u$ Ni \& Sacks \cite{NS} proved non-uniqueness for the critical value of $p$ (Theorem 3); see also Matos \& Terraneo \cite{MatosTerr} and Haraux \& Weissler \cite{HW}. It would be interesting to see whether it is possible to obtain an exact characterisation of those $f$ that admit unique solutions, perhaps based on asymptotic conditions generalising (\ref{Lip}) in the way that our conditions for local existence generalise the growth rates of the canonical example $f(u)=|u|^{p-1}u$.

\section*{Acknowledgments}

JCR is supported by an EPSRC Leadership Fellowship EP/G007470/1. The work of MS was supported by an EPSRC Platform Grant EP/I019138/1, and also by EP/G007470/1. The work leading to this paper was begun during a visit to the University of Zurich by JCR as a guest of Michel Chipot, and JCR would like to thank Michel for his hospitality. The final version profited greatly from visits to Warwick by RL and AVL, funded by EP/G007470/1.

\bibliographystyle{model1-num-names}
\bibliography{<your-bib-database>}

\end{document}